\documentclass
{amsart}
\usepackage{graphicx}
\newtheorem{thm}{Theorem}[section]
\theoremstyle{definition}
\newtheorem{cor}[thm]{Corollary}
\newtheorem{prop}[thm]{Proposition}
\newtheorem{defn}[thm]{Definition}
\newtheorem{lem}[thm]{Lemma}

\newtheorem{ex}[thm]{Example}

\numberwithin{equation}{section}
\begin{document}
\title[2-irreducible and strongly 2-irreducible submodules of a module]
{2-irreducible and strongly 2-irreducible submodules of a module}
\author{F. Farshadifar*}
\address{\llap{*\,} (Corresponding Author) Department of Mathematics, Farhangian University, Tehran, Iran.}
\email{f.farshadifar@cfu.ac.ir}

\author{H. Ansari-Toroghy**}
\address{\llap{**\,}Department of pure Mathematics\\
Faculty of mathematical
Sciences\\
University of Guilan\\
P. O. Box 41335-19141, Rasht, Iran.}
\email{ansari@guilan.ac.ir}

\begin{abstract}
Let $R$ be a commutative ring with identity and $M$ be an $R$-module. In this paper, we will introduce the concept of 2-irreducible (resp., strongly 2-irreducible) submodules of $M$ as a generalization of irreducible (resp., strongly irreducible) submodules of $M$ and investigated some properties of these classes of modules.
\end{abstract}
\maketitle
\section{Introduction}
\noindent
Throughout this paper, $R$ will denote a commutative ring with
identity and $\Bbb Z$ will denote the ring of integers.

An ideal $I$ of $R$ is said to be \emph{irreducible} if $I = J_1\cap J_2$ for ideals $J_1$ and
$J_2$ of $R$ implies that either $I=J_1$ or $I=J_2$. A proper ideal $I$ of $R$ is said to be
\emph{strongly irreducible} if for ideals $J_1, J_2$ of $R$, $J_1\cap J_2 \subseteq I$ implies that $J_1 \subseteq I$ or
$J_2 \subseteq I$ \cite{HRR02}. An ideal $I$ of $R$ is said to be \emph{2-irreducible} if whenever $I = J_1 \cap J_2 \cap J_3$ for ideals $J_1, J_1$ and $J_3$ of $R$, then
either $I = J_1 \cap J_2$ or $I= J_1 \cap J_3$ or $I = J_2 \cap J_3$. Clearly, any irreducible ideal is a
2-irreducible ideal \cite{DM16}.

A proper submodule $N$ of an $R$-module $M$ is said to be \emph{irreducible} (resp., \emph{strongly irreducible}) if for submodules $H_1$ and $H_2$ of $M$, $N=H_1\cap H_2$ (resp., $H_1\cap H_2 \subseteq N$) implies that  $N=H_1$ or $N=H_2$.( resp., $H_1 \subseteq N$ or $H_2 \subseteq N$). 

The main purpose of this paper is to introduce the concept of 2-irreducible and strongly 2-irreducible submodules of an $R$-module $M$ as a generalization of irreducible and strongly irreducible submodules of $M$ and obtain some related results.

A submodule $N$ of an $R$-module $M$ is said to be a \textit{2-irreducible submodule} if whenever $N=H_1\cap H_2\cap H_3$ for submodules $H_1$, $H_2$ and $H_3$ of $M$, then either $N=H_1 \cap H_2$ or $N=H_2\cap H_3$ or $N=H_1\cap H_3$ (Definition \ref{11ld.1}).

A proper submodule $N$ of an $R$-module $M$ is said to be a \textit{strongly 2-irreducible submodule} if whenever $H_1\cap H_2\cap H_3 \subseteq N$ for submodules $H_1$, $H_2$ and $H_3$ of $M$, then either $H_1 \cap H_2 \subseteq N$ or $H_2\cap H_3 \subseteq N$ or $H_1\cap H_3 \subseteq N$  (Definition \ref{111ld.1}).

In Section 2 of this paper, for an $R$-module $M$, among other results, we prove that if $M$ is a Noetherian $R$-module and $N$ is a 2-irreducible submodule of $M$, then either $N$ is irreducible or $N$ is an intersection of exactly two irreducible submodules of $M$ (Theorem \ref{t1.5}). In Theorem \ref{t2.4}, we provide a characterization for strongly 2-irreducible submodules of $M$. Also, it is shown that if $M$ is a strong comultiplication $R$-module, then every non-zero proper submodule of $R$ is a  strongly sum 2-irreducible $R$-module if and only if every non-zero proper submodule of $M$ is a  strongly 2-irreducible submodule of $M$ (Theorem \ref{t42.1}).
Further, it is proved that if $N$ is a submodule of a finitely generated multiplication $R$-module $M$, then $N$ is a strongly 2-irreducible submodule of $M$ if and only if $(N:_RM)$ is a strongly 2-irreducible ideal of $R$ (Theorem \ref{p1.2}).
In Theorem \ref{t1.9} and \ref{tt1.9}, we provide some useful characterizations for strongly 2-irreducible submodules of some special classes of modules. Example \ref{e1.2} shows that the concepts of strongly irreducible submodules and strongly 2-irreducible submodules are different in general. Finally, let $R = R_1\times R_2\times \cdots \times R_n$ ($2\leq n < \infty$) be a decomposable ring and $M =M_1 \times M_2 \cdots \times M_n$ be an $R$-module, where for every $1\leq i \leq n$, $M_i$ is an $R_i$-module, respectively,  it is proved that a proper submodule $N$ of $M$ is a strongly 2-irreducible submodule of $M$ if and only if either $N =\times^n_{i=1}N_i$ such that for some $k \in \{1, 2, ..., n\}$, $N_k$ is a strongly 2-irreducible submodule of $M_k$, and $N_i =M_i$ for every $i\in \{1, 2, ..., n\}\setminus \{k\}$ or $N =\times^n_{i=1}N_i$ such that for some $k,m \in \{1, 2, ..., n\}$, $N_k$ is a strongly irreducible submodule of $M_k$, $N_m$ is a strongly irreducible submodule of $M_m$, and $N_i =M_i$ for every $i \in \{1, 2, ..., n\}\setminus \{k,m\}$ (Theorem \ref{t11.16}).

\section{Main results}
\begin{defn}\label{11ld.1}
We say that a submodule $N$ of an $R$-module $M$ is a \textit{2-irreducible submodule} if whenever $N=H_1\cap H_2\cap H_3$ for submodules $H_1$, $H_2$ and $H_3$ of $M$, then either $N=H_1 \cap H_2$ or $N=H_2\cap H_3$ or $N=H_1\cap H_3$.
\end{defn}

\begin{thm}\label{t1.5}
Let $M$ be a Noetherian $R$-module. If $N$ is a 2-irreducible submodule of $M$, then either $N$ is irreducible or $N$ is an intersection of exactly two irreducible submodules of $M$.
\end{thm}
\begin{proof}
Let $N$ be a 2-irreducible submodule of $M$. By \cite[Exercise 9.31]{SH90}, $N$ can be written
as a finite irredundant irreducible decomposition $N=N_1\cap N_2 \cap ... \cap N_k$. We show that
either $k=1$ or $k=2$. If $k > 3$, then since $N$ is 2-irreducible, $N=N_i \cap N_j$ for some
$1 \leq i,j \leq k$, say $i=1$ and $j=2$. Therefore $N_1 \cap N_2\subseteq N_3 $, which is a contradiction.
\end{proof}

\begin{cor}\label{c1.6}
Let $M$ be a Noetherian multiplication $R$-module. If $N$ is a 2-irreducible submodule of $M$, then $N$ a 2-absorbing primary  submodule of $M$.
\end{cor}
\begin{proof}
Let $N$ be a 2-irreducible submodule of $M$. By the fact that every irreducible submodule of a Noetherian $R$-module is primary and regarding Theorem \ref{t1.5}, we have either $N$ is a primary submodule or is a sum of two primary submodules. It is clear that every primary submodule is 2-absorbing  primary, also the sum of two primary submodules is a 2-absorbing
primary submodule, by \cite[Theorem 2.20]{mtoa16}.
 \end{proof}

\begin{defn}\label{111ld.1}
We say that a proper submodule $N$ of an $R$-module $M$ is a \textit{strongly 2-irreducible submodule} if whenever $H_1\cap H_2\cap H_3 \subseteq N$ for submodules $H_1$, $H_2$ and $H_3$ of $M$, then either $H_1 \cap H_2 \subseteq N$ or $H_2\cap H_3 \subseteq N$ or $H_1\cap H_3 \subseteq N$.
\end{defn}

\begin{ex}\label{e11ld.1}\cite[Corollary 2]{DM16}
Consider the $\Bbb Z$-module $\Bbb Z$. Then $n \Bbb Z$ is a strongly 2-irreducible submodule of  $\Bbb Z$ if $n=0$, $p^t$ or $p^rq^s$, where $p, q$ are prime integers and $t, r, s$ are natural numbers.
\end{ex}

\begin{prop}\label{p22.4}
The strongly 2-irreducible submodules of a distributive $R$-module are precisely the 2-irreducible submodules.
\end{prop}
\begin{proof}
This is straightforward.
\end{proof}

\begin{thm}\label{t2.4}
Let $N$ be a proper submodule of an $R$-module $M$. Then the following conditions are equivalent:
\begin{itemize}
\item [(a)] $N$ is a strongly 2-irreducible submodule;
\item [(b)] For all elements $x, y, z$ of $M$, we have
$(Rx + Ry) \cap (Rx + Rz) \cap (Ry+Rz)\subseteq N$ implies that either $(Rx + Ry) \cap (Rx + Rz)\subseteq N$ or  $(Rx + Ry) \cap (Ry+Rz)\subseteq N$ or $(Rx + Rz) \cap (Ry+Rz)\subseteq N$.
\end{itemize}
\end{thm}
\begin{proof}
$(a) \Rightarrow (b)$
This ia clear.

$(b) \Rightarrow (a)$
Let  $H_1\cap H_2\cap H_3 \subseteq N$ for submodules $H_1$, $H_2$ and $H_3$ of $M$. If  $H_1\cap H_2\not \subseteq N$,  $H_1\cap H_3 \not \subseteq N$, and  $H_2\cap H_3 \not \subseteq N$, then there exist  elements $x, y, z$ of $M$ such that $x \in H_2\cap H_3 $,  $y \in H_1\cap H_3$, and  $z \in H_1 \cap H_2$ but  $x \not \in N$, $y \not \in N$, and $z\not \in N$. Therefore,
$$
(Ry + Rz) \cap (Rx + Rz)\cap (Rx + Ry)\subseteq H_1\cap H_2 \cap H_3\subseteq  N.
$$
Hence by the part (a),  either $(Ry + Rz) \cap (Rx + Rz) \subseteq N$ or   $(Ry + Rz) \cap (Rx + Ry)\subseteq N$ or  $(Rx + Rz)\cap (Rx + Ry)\subseteq N$. Thus either $z \in N$ or $y \in N$ or $x \in N$. This contradiction completes the proof.
\end{proof}

Recall that a \textit{waist submodule} of an $R$-module $M$ is a submodule that is comparable to any other submodules of $M$.
\begin{prop}\label{p2.4}
Let $N$ be a proper submodule of an $R$-module $M$. Then we have the following.
\begin{itemize}
\item [(a)] If $N$ is a strongly 2-irreducible submodule of $M$, then it is also a 2-irreducible submodule of $M$.
\item [(b)] If $N$ is a strongly 2-irreducible submodule of $M$, then $N$ is a strongly 2-irreducible submodule of $T$ and $N/K$ is a strongly 2-irreducible submodule of $M/K$ for any $K \subseteq N \subseteq T$.
\item [(c)] If for all elements  $x, y, z$ of $M$ we have
$Rx \cap Ry \cap Rz \subseteq N$ implies that either $Rx \cap Ry \subseteq N$  or $Rx \cap Rz \subseteq N$  or $Ry \cap Rz \subseteq N$, then $N$ is a strongly 2-irreducible submodule of $M$.
\item [(d)] If $N$ is a waist submodule of $M$, then $N$ is strongly 2-irreducible submodule of $M$ if and only if $N$ is 2-irreducible module.
\item [(e)] If $N$ satisfies $(N + T) \cap (N +K) = N+ (T \cap K)$, whenever $T \cap K \subseteq N$, then $N$ is
strongly 2-irreducible submodule of $M$ if and only if $N$ is a 2-irreducible module.
\end{itemize}
\end{prop}

\begin{proof}
(a) Let $N$ be a strongly 2-irreducible submodule of $M$ and let $N=H_1\cap H_2 \cap H_3$ for submodules $H_1$, $H_2$ and $H_3$ of $M$. Then by assumption, either $H_1\cap H_2\subseteq N$ or  $H_1\cap H_3\subseteq N$ or  $H_2 \cap H_3\subseteq N$. Now the result follows from the fact that the reverse of inclusions are clear.

The parts (b), (d), and (e) are straightforward.

(c) Let  $ H_1 \cap H_2 \cap H_3 \subseteq N$ for submodules $H_1$, $H_2$ and $H_3$ of $M$. If  $ H_1 \cap H_2 \not \subseteq N$,  $ H_1 \cap H_3 \not \subseteq N$, and  $H_2 \cap H_3 \not \subseteq N$, then there exist elements $x, y, z$ of $M$ such that $x \in H_2\cap H_3$,  $y \in H_1\cap H_3$, and  $z \in H_1\cap H_2$ but  $x \not \in N$, $y \not \in N$, and $z \not \in N$. Now the result follows by assumption.
\end{proof}

An $R$-module $M$ is said to be a \emph{comultiplication module} if for every submodule $N$ of $M$ there exists an ideal $I$ of $R$ such that $N=(0:_MI)$, equivalently, for each submodule $N$ of $M$, we have $N=(0:_MAnn_R(N))$ \cite{AF07}.

An $R$-module $M$ satisfies the \emph{double annihilator
conditions} (DAC for short)  if for each ideal $I$ of $R$
we have $I=Ann_R(0:_MI)$ \cite{Fa95}.

An $R$-module $M$ is said to be a \emph{strong comultiplication module} if $M$ is
a comultiplication $R$-module and satisfies the DAC conditions \cite{AF09}.

A submodule $N$ of an $R$-module $M$ is said to be a \textit{strongly sum 2-irreducible submodule} if whenever $N\subseteq H_1+H_2+H_3$ for submodules $H_1$, $H_2$ and $H_3$ of $M$, then either $N\subseteq H_1+H_2$ or $N\subseteq H_2+H_3$ or $N\subseteq H_1+H_3$. Also, $M$ is said to be a \textit{strongly sum 2-irreducible module} if and only if $M$ is a strongly sum 2-irreducible submodule of itself \cite{FA05}.
\begin{thm}\label{t42.1}
Let $M$ be a strong comultiplication $R$-module. Then every non-zero proper submodule of $R$ is a strongly sum 2-irreducible $R$-module if and only if every non-zero proper submodule of $M$ is a  strongly 2-irreducible submodule of $M$.
\end{thm}
\begin{proof}
$"\Rightarrow"$
Let $N$ be a non-zero proper submodule of $M$ and let $H_1\cap H_2\cap H_3\subseteq N$ for submodules $H_1$, $H_2$ and $H_3$ of $M$. Then by using \cite[2.5]{HH19},
$$
Ann_R(N) \subseteq Ann_R(H_1)+Ann_R(H_2)+Ann_R(H_3).
$$
 This implies that either $Ann_R(N) \subseteq Ann_R(H_1)+Ann_R(H_2)$ or $Ann_R(N) \subseteq Ann_R(H_1)+Ann_R(H_3)$ or $Ann_R(N) \subseteq Ann_R(H_2)+Ann_R(H_3)$ since by assumption, $Ann_R(N)$ is a strongly sum 2-irreducible $R$-module. Therefore, either $H_1\cap H_2\subseteq N$ or $H_1\cap H_3\subseteq N$ or $H_2\cap H_3\subseteq N$ since $M$ is a comultiplication $R$-module.

$"\Leftarrow"$  Let $I$ be a non-zero proper submodule of $R$ and let $I \subseteq I_1+I_2+I_3$. Then
$$
(0:_MI_1)\cap (0:_MI_2)\cap (0:_MI_3) \subseteq (0:_MI).
$$
Thus by assumption,  either $(0:_MI_1)\cap (0:_MI_2) \subseteq (0:_MI)$ or $(0:_MI_1)\cap (0:_MI_3) \subseteq (0:_MI)$ or $(0:_MI_2)\cap (0:_MI_3) \subseteq (0:_MI)$. This implies that either $(0:_MI_1+I_2) \subseteq (0:_MI)$ or $(0:_MI_1+I_3) \subseteq (0:_MI)$ or $(0:_MI_2+I_3) \subseteq (0:_MI)$. Thus either $I \subseteq I_1+I_2$ or $I \subseteq I_1+I_3$ or $I \subseteq I_2+I_3$ since $M$ is a strong comultiplication $R$-module.
\end{proof}

An $R$-module $M$ is said to be a \emph{multiplication module} if for every submodule $N$ of $M$ there exists an ideal $I$ of $R$ such that $N=IM$ \cite{Ba81}.

\begin{thm}\label{p1.2}
Let $N$ be a submodule of a finitely generated multiplication $R$-module $M$. Then $N$ is a strongly 2-irreducible submodule of $M$ if and only if $(N:_RM)$ is a strongly 2-irreducible ideal of $R$.
\end{thm}
\begin{proof}
$"\Rightarrow"$ Let $N$ be a strongly  2-irreducible submodule of $M$ and let $ J_1\cap J_2 \cap J_3 \subseteq (N:_RM)$ for some ideals $J_1, J_2$, and $J_3$ of $R$. Then
$$
J_1M\cap J_2M \cap J_3M \subseteq (N:_RM)M=N
$$
 by \cite[Corollary 1.7]{BS89}. Thus by assumption, either  $J_1M\cap J_2M \subseteq N$ or  $J_1M\cap J_3M \subseteq N$ or $J_2M \cap J_3M \subseteq N$. Hence, either
$ (J_1\cap J_2)M \subseteq (N:_RM)M$ or $ (J_1\cap J_3)M \subseteq (N:_RM)M$ or $(J_2 \cap J_3)M \subseteq (N:_RM)M$.
Therefore, either $ J_1\cap J_2\subseteq (N:_RM)$ or  $ J_1\cap J_3 \subseteq (N:_RM)$ or  $J_2 \cap J_3 \subseteq (N:_RM)$ by \cite[Corollary of Theorem 9]{SM88}.

$"\Leftarrow"$
Let $(N:_RM)$ is a strongly 2-irreducible ideal of $R$ and let
$H_1\cap H_2 \cap H_3\subseteq N$ for some submodules $H_1$, $H_2$ and $H_3$ of $M$. Then we have
$$
(H_1\cap H_2 \cap H_3:_RM)M=((H_1:_RM)\cap (H_2 :_RM)\cap (H_3:_RM))M\subseteq (N:_RM)M.
$$
Thus $(H_1:_RM)\cap (H_2 :_RM)\cap (H_3:_RM)\subseteq (N:_RM)$
by \cite[Corollary of Theorem 9]{SM88}. Hence,  either  $(H_1:_RM)\cap (H_2 :_RM)\subseteq (N:_RM)$ or  $(H_1:_RM)\cap (H_3:_RM))\subseteq (N:_RM)$ or  $(H_2 :_RM)\cap (H_3:_RM)\subseteq (N:_RM)$
since $(N:_RM)$ is a strongly 2-irreducible ideal of $R$.
Therefore, either $H_1\cap H_2\subseteq N$ or $H_1\cap H_3\subseteq N$ or $H_2 \cap H_3\subseteq N$
 by \cite[Corollary 1.7]{BS89}.
\end{proof}

\begin{ex}
Consider the $\Bbb Z$-module $\Bbb Z_{p^tq^nr^m}$, where $p, q, r$ are prime integers and $t,n, m$ are natural numbers. 
\begin{itemize}
\item [(a)] By using Theorem \ref{p1.2} and Example \ref{e11ld.1}, one can see that $\bar{p^t}\Bbb Z_{p^tq^nr^m}$ and $\bar{q^nr^m}\Bbb Z_{p^tq^nr^m}$ are strongly 2-irreducible submodules of  $\Bbb Z_{p^tq^nr^m}$.
\item [(b)] $\bar{pqr}\Bbb Z_{p^3qr}=\bar{pq}Z_{p^3qr} \cap \bar{pr}\Bbb Z_{p^3qr} \cap \bar{qr}\Bbb Z_{p^3qr}$ implies that $\bar{pqr}\Bbb Z_{p^3qr}$ is not  a  2-irreducible submodule of  $\Bbb Z_{p^3qr}$.
\end{itemize}
\end{ex}

The following example shows that the concepts of strongly irreducible submodules and strongly 2-irreducible submodules are different in general.
\begin{ex}\label{e1.2}
Consider the $\Bbb Z$-module $\Bbb Z_6$.  Then $0=\bar{3}\Bbb Z_6 \cap \bar{2}\Bbb Z_6$ implies that the $0$  submodule of $\Bbb Z_6$ is not strongly irreducible. But $(0:_{\Bbb Z}\Bbb Z_6)=6\Bbb Z$ is a strongly 2-irreducible ideal of $\Bbb Z$ by Example \ref{e11ld.1}. Since the $\Bbb Z$-module $\Bbb Z_6$  is a finitely generated multiplication $\Bbb Z$-module, $0$ is a strongly 2-irreducible submodule of $\Bbb Z_6$ by Theorem \ref{p1.2}.
\end{ex}

\begin{lem}\label{p1.4}
Let $M$ be an $R$-module. If $N_1$ and $N_2$ are strongly irreducible submodules of $M$, then $N_1\cap N_2$ is a  strongly
2-irreducible submodule of $M$.
\end{lem}
\begin{proof}
This is straightforward.
\end{proof}

A proper submodule $P$ of an $R$-module $M$ is said to be \emph{prime} if for any $r \in R$ and $m \in M$ with $rm \in P$, we have $m \in P$ or $r \in (P:_RM)$ \cite{Da78}.
\begin{prop}\label{p1.7}
Let $M$ be a multiplication $R$-module and let $N_1$, $N_2$, and $N_3$ be prime submodules of $M$ such that $N_1 + N_2=N_1 +N_3=N_2 + N_3=M$. Then $N_1\cap N_2 \cap N_3$ is not a  strongly 2-irreducible submodule of $M$.
\end{prop}
\begin{proof}
Assume on the contrary that $N_1\cap N_2 \cap N_3$ is a  strongly 2-irreducible submodule of $M$. Then $N_1\cap N_2 \cap N_3\subseteq N_1\cap N_2 \cap N_3$ implies that either $N_1\cap N_2 \subseteq N_1\cap N_2 \cap N_3$ or $N_1\cap N_3\subseteq N_1\cap N_2 \cap N_3$ or $N_2 \cap N_3\subseteq N_1\cap N_2 \cap N_3$. We can assume without loss of generality that $N_1\cap N_2 \subseteq N_1\cap N_2 \cap N_3$. Then $N_1\cap N_2 \subseteq N_3$. It follows that $(N_1:_RM)N_2 \subseteq N_3$. As $N_3$ is a prime submodule of $M$, we have  $N_2 \subseteq N_3$ or $(N_2:_RM) \subseteq (N_3:_RM)$. Thus $N_2 \subseteq N_3$ or $N_1 \subseteq N_3$ since $M$ is a multiplication $R$-module. Therefore, $N_3=M$, which is a contradiction.
\end{proof}

\begin{cor}\label{c1.8}
Let $M$ be a multiplication $R$-module such that every proper submodule of $M$ is  strongly 2-irreducible. Then $M$ has at most two maximal submodules.
\end{cor}
\begin{proof}
This follows from Proposition \ref{p1.7}
\end{proof}

Let $N$ be a submodule of an $R$-module $M$. The intersection of all prime submodules of $M$ containing $N$ is said to be the (\emph{prime}) \emph{radical} of $N$ and denote by $rad_MN$ (or simply by $rad(N)$). In case $N$ does not contained in any prime submodule, the radical of $N$ is defined to be $M$.
Also, $N \not =M$ is said to be a \emph{radical submodule} of $M$ if $rad(N)=N$ \cite{MM86}

\begin{lem}\label{l11.9}
Let $I$ be an ideal of $R$ and $N$ be a submodule of an $R$-module $M$. Then $rad(IN)=rad(N) \cap rad(IM)$.
\end{lem}
\begin{proof}
By \cite [Corollary of Theorem 6]{Lu89}, we have $rad(N\cap IM))= rad(N)\cap rad(IM)$. Since $IN \subseteq IM \cap N$, $rad(IN) \subseteq rad(IM \cap N)$. Thus $rad(IN) \subseteq rad(N)\cap rad(IM)$. Now let $P$ be a prime submodule of $M$ such that $IN\subseteq P$. As $P$ is prime, $N \subseteq P$ or $I \subseteq (P:_RM)$. Hence  $N \cap IM \subseteq P$. This in tourn implies that $rad(N)\cap rad(IM)\subseteq  rad(IN) $, as desired.
\end{proof}

A proper ideal $I$ of $R$ is said to be a \emph{2-absorbing ideal}
of $R$ if whenever $a, b, c \in R$ and $abc \in I$, then $ab \in I$ or
 $ac \in I$ or $bc \in I$ \cite{Ba07}.

A proper submodule $N$ of an $R$-module $M$ is said to be a \emph{2-absorbing primary submodule} of $M$ if whenever $a, b \in R$, $m \in M$, and $abm \in N$, then $am \in rad(N)$ or $bm \in rad(N)$  or $ab \in (N :_R M)$ \cite{mtoa16}.

A proper submodule $N$ of an $R$-module $M$ is called a \emph{2-absorbing submodule }of
 $M$ if whenever $abm \in N$
for some $a, b \in R$ and $m \in M$, then $am \in N$ or $bm \in N$ or
$ab \in (N :_R M)$ \cite{YS11} and  \cite{pb12}.

\begin{thm}\label{t1.9}
Let $M$ be a finitely generated multiplication $R$-module and $N$ be a radical submodule of $M$. Then the following
conditions are equivalent:
\begin{itemize}
\item [(a)] $N$ is a strongly 2-irreducible submodule of $M$;
\item [(b)] $N$ is a 2-absorbing submodule of $M$;
\item [(c)] $N$ is a 2-absorbing primary submodule of $M$;
\item [(d)] $N$ is either a prime submodule of $M$ or is an intersection of exactly two prime submodules of $M$.
\end{itemize}
\end{thm}
\begin{proof}
$(a) \Rightarrow (b)$
Let $I$, $J$ be ideals of $R$ and $K$ be a submodule of $M$ such that $IJK \subseteq N$. Then
 by using  Lemma \ref{l11.9},
$$
 K\cap IM \cap JM \subseteq rad(K)\cap rad(IM)\cap  rad(JM)=rad(IJK) \subseteq rad(N)=N
$$
Hence by part (a), either $ K\cap IM \subseteq N$ or $ K\cap JM \subseteq N$ or $IM \cap JM \subseteq N$. Thus either $IK\subseteq N$ or $JK\subseteq N$ or $IJM \subseteq N$ as needed.

$(b) \Rightarrow(c)$
This is clear.

$(c) \Rightarrow (b)$
This is clear by using \cite[Theorem 2.6]{mtoa16}.

$(b) \Rightarrow (d)$
Since $N$ is a 2-absorbing submodule of $M$,  $(N:_RM)$ is a 2-absorbing ideal of $R$ by \cite[Proposition 1]{DS12}.
Hence $\sqrt{(N:_RM)} = P$ is a prime ideal of $R$  or $\sqrt{(N:_RM)} =  P \cap Q$, where $P$ and $Q$ are distinct
prime ideals of $R$ that are minimal over $(N:_RM)$ by \cite[Theorem 2.4]{Ba07}.
We have $\sqrt{(N:_RM)}M=rad(N)$ by \cite[Theorem 4]{MM86}.
If $\sqrt{(N:_RM)} = P$, then $PM=rad(N)$. Since $M$ is a multiplication $R$-module, $PM$ is a prime submodule of $M$ by \cite[Corollary 2.11]{BS89}. Now let $\sqrt{Ann_R(N)} =  P \cap Q$, where $P$ and $Q$ are distinct
prime ideals of $R$. Then $(P\cap Q)M=rad(N)$. By \cite[Corollary 1.7]{BS89}, $(P\cap Q)M=PM\cap QM$.  Thus  by \cite[Corollary 2.11]{BS89}, $rad(N)$ is an intersection of two prime submodules of $M$. Now, we prove the claim by
assumption that $N$ is a radical submodule of $M$.

$(d) \Rightarrow (a)$
This follows from Lemma \ref{p1.4}.
\end{proof}

The following example shows that parts $(a)$ and $(b)$ of Theorem \ref{t1.9} are not equivalent in general.
\begin{ex}\label{e11.2}
Consider the submodule $G_t=\langle 1/p^t+\Bbb Z\rangle$ of the $\Bbb Z$-module $\Bbb Z_{p^\infty}$.  Then $G_t$ is a strongly 2-irreducible submodule of $\Bbb Z_{p^\infty}$.  But  $G_t$ is not a 2-absorbing submodule of $\Bbb Z_{p^\infty}$. It should be note that the $\Bbb Z$-module $\Bbb Z_{p^\infty}$ is not a finitely genrated multiplication $\Bbb Z$-module.
\end{ex}

A submodule $N$ of an $R$-module $M$ is said to be \emph{pure} if $IN=IM \cap N$ for every ideal $I$ of $R$ \cite{AF74}.
Also, an $R$-module $M$ is said to be \emph{fully pure} if every submodule of $M$ is pure \cite{AF122}.
\begin{thm}\label{tt1.9}
Let $M$ be a fully pure multiplication $R$-module and $N$ be a submodule of $M$. Then the following
conditions are equivalent:
\begin{itemize}
\item [(a)] $N$ is a strongly 2-irreducible submodule of $M$;
\item [(b)] $N$ is a 2-absorbing submodule of $M$;
\item [(c)] $N$ is a 2-irreducible submodule of $M$.
\end{itemize}
\end{thm}
\begin{proof}
$(a) \Rightarrow (b)$
Let $I$, $J$ be ideals of $R$ and $K$ be a submodule of $M$ such that $IJK \subseteq N$. Then since $M$ is fully pure,
$$
K \cap IM \cap JM=IJK \subseteq N.
$$
Hence by part (a), either $K \cap IM\subseteq N$ or $K \cap JM\subseteq N$ or $IM \cap JM\subseteq N$. Thus either $IK\subseteq N$ or $JK\subseteq N$ or $IJM \subseteq N$.

$(b) \Rightarrow (a)$
Let $H_1\cap H_2 \cap H_3\subseteq N$ for submodules $H_1$, $H_2$ and $H_3$ of $M$. Then
$$
(H_1:_RM)\cap (H_2:_RM) \cap (H_3:_RM)=(H_1\cap H_2 \cap H_3:_RM)\subseteq (N:_RM).
$$
Thus either $(H_1:_RM)(H_2:_RM)\subseteq (N:_RM)$ or  $(H_1:_RM)(H_3:_RM)\subseteq (N:_RM)$ or  $(H_2:_RM)(H_3:_RM)\subseteq (N:_RM)$ since $(N:_RM)$ is a 2-absorbing ideal of $R$ by \cite[Proposition 1]{DS12}.
We can assume without loss of generality that  $(H_1:_RM)(H_2:_RM)\subseteq (N:_RM)$. Thus as $M$ is fully pure, we have
$$
 (H_1:_RM)M \cap (H_2:_RM)M\subseteq (N:_RM)M \subseteq N.
$$
Therefore, $H_1 \cap H_2 \subseteq N$ since $M$ is a multiplication $R$-module.

$(a) \Leftrightarrow (c)$
By \cite[proof of Theorem 2.19]{AF122}, $M$ is a distributive $R$-module. Now the result follows from Proposition \ref{p22.4}.
\end{proof}

\begin{lem}\label{t1.5}
Let $M$ be an $R$-module, $S$ a multiplicatively closed subset of $R$, and $N$ be a finitely generated submodule of $M$. If $S^{-1}N \subseteq S^{-1}K$ for a submodule $K$ of $M$, then there exists $s \in S$ such that $sN \subseteq K$.
\end{lem}
\begin{proof}
 This is straightforward.
\end{proof}

\begin{prop}\label{p1.6}
Let $M$ be an $R$-module, $S$ be a multiplicatively closed subset of $R$ and $N$ be a finitely generated prime strongly 2-irreducible submodule of $M$ such that $(N:_RM) \cap S=\emptyset$. Then $S^{-1}N$ is a strongly 2-irreducible submodule of $S^{-1}M$ if $S^{-1}N\not =S^{-1}M$.
\end{prop}
\begin{proof}
Let $S^{-1}H_1\cap S^{-1}H_2 \cap S^{-1}H_3\subseteq S^{-1}N$ for submodules $S^{-1}H_1$, $S^{-1}H_2$ and $S^{-1}H_3$ of $S^{-1}M$. Then
$S^{-1}(H_1\cap H_2 \cap H_3)\subseteq S^{-1}N$.  By Lemma \ref{t1.5}, there exists $s \in S$ such that $s (H_1 \cap H_2 \cap H_3) \subseteq N$. This implies that $H_1 \cap H_2 \cap H_3 \subseteq N$ since $N$ is prime and $(N:_RM) \cap S=\emptyset$. Now as $N$ is a strongly 2-irreducible submodule of $M$, we have either $H_1 \cap H_2 \subseteq N$ or $H_1 \cap H_3 \subseteq N$ or $H_2 \cap H_3 \subseteq N$. Therefore, either  $S^{-1}H_1\cap S^{-1}H_2\subseteq S^{-1}N$ or $S^{-1}H_1\cap S^{-1}H_3\subseteq S^{-1}N$ or $S^{-1}H_2\cap S^{-1}H_3\subseteq S^{-1}N$, as needed.
\end{proof}

\begin{prop}\label{p8.21}
Let $M$ be an $R$-module and $\{K_i\}_{i \in I}$ be a chain of strongly
2-irreducible submodules of $M$. Then $\cap_{i \in I}K_i$ is a strongly 2-irreducible submodule of $M$.
\end{prop}
\begin{proof}
Let $H_1\cap H_2 \cap H_3\subseteq \cap_{i \in I}K_i$ for submodules $H_1$, $H_2$ and $H_3$ of $M$.
Assume that $H_1+H_2\not\subseteq \cap_{i \in I}K_i$, $H_1+H_3\not\subseteq \cap_{i \in I}K_i$, and $H_2+H_3\not\subseteq \cap_{i \in I}K_i$. Then there are $m,n,t \in I$, where
$ H_1\cap H_2 \not \subseteq K_m$, $H_1\cap H_3 \not\subseteq K_n$, and $H_2\cap H_3 \not\subseteq K_t$.  Since  $\{K_i\}_{i \in I}$ is a chain we can assume that $K_m \subseteq K_n\subseteq K_t$. But as  $H_1\cap H_2 \cap H_3\subseteq K_m$ and $K_m$ is a strongly sum 2-irreducible submodule of $M$, we have either $H_1\cap H_2 \subseteq K_m$ or  $H_1\cap H_3\subseteq K_m$ or $H_2 \cap H_3\subseteq K_m$.
In any case, we get a contradiction.
\end{proof}

\begin{thm}\label{t11.12}
Let $f : M \rightarrow \acute{M}$ be a epimorphism of R-modules. Then we have the following.
\begin{itemize}
\item [(a)] If $N$ is a strongly 2-irreducible submodule of $M$ such that $ker(f) \subseteq N$, then $f(N)$ is a  strongly 2-irreducible submodule  of $f(M)$.
\item [(b)] If $\acute{N}$ is a strongly 2-irreducible submodule of $\acute{M}$, then $f^{-1}(\acute{N})$ is a strongly 2-irreducible submodule of $M$.
 \end{itemize}
\end{thm}
\begin{proof}
(a) Let $N$ be a strongly 2-irreducible submodule of $M$. If $f(N) =\acute{M}$, then we have
$N+Ker(f)=f^{-1}(f(N))=f^{-1}(\acute{M})=f^{-1}(f(M))=M$. Now as $ker(f) \subseteq N$, we get that $N=M$, which is a contradiction. Therefore, $f(N) \not=\acute{M}$. Suppose that $\acute{H_1}\cap \acute{H_2}\cap \acute{H_3}\subseteq f(N)$ for submodules $\acute{H_1}$, $\acute{H_2}$ and $\acute{H_3}$ of $f(M)$. Then
$ f^{-1}(\acute{H_1}) \cap f^{-1}(\acute{H_2}) \cap f^{-1}(\acute{H_3})\subseteq f^{-1}(f(N))=N$ since $ker(f) \subseteq N$. Thus by assumption, either $ f^{-1}(\acute{H_1}) \cap f^{-1}(\acute{H_2})\subseteq N$ or $ f^{-1}(\acute{H_1}) \cap f^{-1}(\acute{H_3})\subseteq N$ or $f^{-1}(\acute{H_2}) \cap f^{-1}(\acute{H_3})\subseteq N$. Now as $f$ is epimorphism,  we have either $\acute{H_1}\cap \acute{H_2}\subseteq f(N)$ or $\acute{H_1}\cap \acute{H_3}\subseteq f(N)$ or $\acute{H_2}\cap \acute{H_3}\subseteq f(N)$, as needed.

(b) Let $\acute{N}$ be a strongly 2-irreducible submodule of $\acute{M}$.
Since $\acute{N} \not =\acute{M}$ and $f$ is a epimorphism, we have $f^{-1}(\acute{N}) \not =M$.
Now let $H_1\cap H_2 \cap H_3\subseteq f^{-1}(\acute{N})$ for submodules $H_1$, $H_2$ and $H_3$ of $M$. Then
$f(H_1)\cap f(H_2) \cap f(H_3)\subseteq f(f^{-1}(\acute{N}))= \acute{N}$. Thus by assumption, either $f(H_1)\cap f(H_2)\subseteq \acute{N}$ or $f(H_1)\cap f(H_3)\subseteq \acute{N}$ or $f(H_2) \cap f(H_3)\subseteq \acute{N}$. Now we have either $H_1\cap H_2\subseteq f^{-1}(\acute{N})$ or $H_1\cap H_3\subseteq f^{-1}(\acute{N})$ or $H_2 \cap H_3\subseteq f^{-1}(\acute{N})$, as required.
\end{proof}

\begin{thm}\label{t222.1}
Let $M$ be a finitely generated multiplication distributive $R$-module and let $N$ be a non-zero proper submodule of $M$. Then the following statements are equivalent:
\begin{itemize}
\item [(a)] $N$ is a strongly 2-irreducible submodule of $M$;
\item [(b)] $(N:_RM)$ is a strongly 2-irreducible ideal of $R$;
\item [(c)] $(N:_RM)$ is a 2-irreducible ideal of $R$.
\end{itemize}
\end{thm}
\begin{proof}
$(a)\Rightarrow (b)$
This follows from Theorem \ref{p1.2}.

$(b)\Rightarrow (c)$
This follows from \cite[Proposition 1]{DM16}.

$(c)\Rightarrow (a)$
Let $H_1\cap H_2 \cap H_3\subseteq N$ for submodules $H_1$, $H_2$ and $H_3$ of $M$. Then as $M$ is a distributive $R$-module, we have
$$
N=N +(H_1\cap H_2\cap H_3)=(N+ H_1)\cap (N\cap H_2)\cap (N\cap H_3).
$$
This implies that
$(N:_RM)=(N+ H_1:_RM) \cap (N+ H_2:_RM) \cap (N+ H_3:_RM)$. Thus by assumption, either$(N:_RM)=(N+ H_1:_RM) \cap (N+ H_2:_RM)$ or $(N:_RM)=(N+ H_1:_RM)\cap (N+ H_3:_RM)$ or $(N:_RM)=(N+ H_2:_RM) \cap (N+ H_3:_RM)$. Therefore, by \cite[Corollary 1.7]{BS89}, either $N=N+(H_1\cap H_2)$ or $N=N+(H_1\cap H_3)$ or $N=N+(H_2\cap H_3)$, since $M$ is a finitely generated multiplication $R$-module. Thus either, $H_1\cap H_2\subseteq N$ or  $H_1\cap H_3\subseteq N$  or  $H_2 \cap H_3\subseteq N$  as needed.
\end{proof}

Let $R_i$ be a commutative ring with identity and $M_i$ be an $R_i$-module, for $i = 1, 2$. Let $R = R_1 \times R_2$. Then $M = M_1 \times M_2$ is an $R$-module and each submodule of $M$ is in the form of $N = N_1 \times N_2$ for some submodules $N_1$ of $M_1$ and $N_2$ of $M_2$.

\begin{thm}\label{t11.15}
Let $R = R_1 \times R_2$ be a decomposable ring and $M = M_1 \times M_2$
be an $R$-module, where $M_1$ is an $R_1$-module and $M_2$ is an $R_2$-module. Suppose that $N = N_1 \times N_2$ is a proper submodule of $M$. Then the following conditions are equivalent:
\begin{itemize}
  \item [(a)] $N$ is a strongly 2-irreducible submodule of $M$;
  \item [(b)] Either $N_1 = M_1$ and $N_2$ strongly 2-irreducible submodule of $M_2$ or $N_2 = M_2$ and $N_1$ is a strongly 2-irreducible submodule of $M_1$ or $N_1$, $N_2$ are strongly irreducible submodules of $M_1$, $M_2$, respectively.
\end{itemize}
\end{thm}
\begin{proof}
$(a) \Rightarrow (b)$.
Let $N = N_1 \times N_2$ be a strongly 2-irreducible submodule of $M$ such that $N_2 = M_2$. From our hypothesis, $N$ is proper, so $N_1 \not =M_1$. Set $\acute{M}=M/(0\times M_2)$. One can see that $\acute{N}=N /(0\times M_2)$ is a strongly 2-irreducible submodule of $\acute{M}$. Also, observe that $\acute{M} \cong M_1$ and $\acute{N} \cong N_1$. Thus $N_1$ is a strongly 2-irreducible submodule of $M_1$. By a similar argument as in the previous case, $N_2$ is a strongly 2-irreducible submodule of $M_2$, where, $N_1 =M_1$. Now suppose that $N_1\not =M_1$ and $N_2 \not =M_2$. We show that $N_1$ is a irreducible submodule of $M_1$. Suppose that $H_1\cap K_1\subseteq N_1$ for some submodules $H_1$ and $K_1$ of $M_1$. Then
$$
(H_1\times M_2) \cap (M_1\times 0) \cap (K_1\times M_2)  \subseteq (H_1\cap K_1)\times 0 \subseteq N_1\times N_2.
$$
Thus by assumption, either $(H_1\times M_2) \cap (M_1\times 0) \subseteq N_1\times N_2$
or $(H_1\times M_2)\cap (K_1\times M_2)  \subseteq N_1\times N_2$ or $(M_1\times 0) \cap (K_1\times M_2)  \subseteq N_1\times N_2$. Therefore, $H_1 \subseteq N_1$ or $K_1 \subseteq N_1$ since $N_2 \not =M_2$. Thus $N_1$ is a strongly irreducible submodule of $M_1$. Similarly, we can show that $N_2$ is strongly irreducible submodule of $M_2$.

$(b) \Rightarrow (a)$.
Suppose that $N = N_1 \times M_2$, where $N_1$ is a strongly 2-irreducible submodule of $M_1$. Then it is clear that $N$ is a strongly 2-irreducible submodule of $M$. Now, assume that $N = N_1 \times N_2$, where $N_1$ and $N_2$ are strongly irreducible  submodules of $M_1$ and $M_2$, respectively. Hence $(N_1 \times M_2)\cap (M_1 \times N_2) = N_1 \times N_2 = N$ is a strongly 2-irreducible submodule of $M$, by Lemma \ref{p1.4}.
\end{proof}

\begin{thm}\label{t11.16}
Let $R = R_1\times R_2\times \cdots \times R_n$ ($2\leq n < \infty$) be a decomposable ring and $M =
M_1 \times M_2 \cdots \times M_n$ be an $R$-module, where for every $1\leq i \leq n$, $M_i$ is an $R_i$-module, respectively. Then for a proper submodule $N$ of $M$ the following conditions are equivalent:
\begin{itemize}
  \item [(a)] $N$ is a strongly 2-irreducible submodule of $M$;
  \item [(b)]  Either $N =\times^n_{i=1}N_i$ such that for some $k \in \{1, 2, ..., n\}$, $N_k$ is a strongly 2-irreducible submodule of $M_k$, and $N_i =M_i$ for every $i\in \{1, 2, ..., n\}\setminus \{k\}$ or $N =\times^n_{i=1}N_i$ such that for some $k,m \in \{1, 2, ..., n\}$, $N_k$ is a strongly irreducible submodule of $M_k$, $N_m$ is a strongly irreducible submodule of $M_m$, and $N_i =M_i$ for every $i \in \{1, 2, ..., n\}\setminus \{k,m\}$.
\end{itemize}
\end{thm}
\begin{proof}
We use induction on $n$. For $n = 2$ the result holds by Theorem \ref{t11.15}. Now
let $3\leq n < \infty$ and suppose that the result is valid when $K = M_1\times \cdots \times M_{n-1}$. We show that the result holds when $M = K \times M_n$. By Theorem \ref{t11.15}, $N$ is a strongly 2-irreducible submodule of $M$ if and only if either $N = L \times M_n$ for some strongly 2-irreducible submodule $L$ of $K$ or $N =K\times L_n$ for some strongly 2-irreducible submodule $L_n$ of $M_n$ or $N = L \times L_n$ for some strongly irreducible submodule $L$ of $K$ and some strongly irreducible submodule $L_n$ of $M_n$. Note that a proper submodule $L$ of $K$ is a strongly irreducible submodule of $K$ if and only
if $L = \times^{n-1}_{i=1}N_i$ such that for some $k \in \{1, 2, ..., n - 1\}$, $N_k$ is a strongly irreducible submodule of $M_k$, and $N_i =M_i$ for every $i\in \{1, 2, ..., n - 1\}\setminus \{k\}$. Consequently the claim is now verified.
\end{proof}

\bibliographystyle{amsplain}

\end{document}